\documentclass[12pt]{amsart}
\usepackage{amscd,amsmath,amsthm,amssymb}
\usepackage{color}
\usepackage{stmaryrd}
\usepackage{tikz}

\def\NZQ{\mathbb}               % the font for N,Z,Q,R,C

\def\ZZ{{\NZQ Z}}

\def\frk{\mathfrak}               % font for "Fraktur"

\def\mm{{\frk m}}

\def\G{{\mathcal G}}

\def\pd{\textup{proj}\phantom{.}\!\textup{dim}}

\def\opn#1#2{\def#1{\operatorname{#2}}} % to make operators
\opn\chara{char} \opn\length{\ell} \opn\pd{pd} \opn\rk{rk}
\opn\projdim{proj\,dim} \opn\injdim{inj\,dim} \opn\rank{rank}
\opn\depth{depth} \opn\grade{grade} \opn\height{height}
\opn\embdim{emb\,dim} \opn\codim{codim}

\opn\Tr{Tr} \opn\bigrank{big\,rank}
\opn\superheight{superheight}\opn\lcm{lcm}
\opn\trdeg{tr\,deg}%\emph{
	\opn\reg{reg} \opn\lreg{lreg} \opn\ini{in} \opn\lpd{lpd}
	\opn\size{size} \opn\sdepth{sdepth}
	\opn\link{link}\opn\fdepth{fdepth}\opn\lex{lex}
	\opn\tr{tr}
	\opn\type{type}
	\opn\gap{gap}
	\opn\diam{diam}
	\opn\Mod{Mod}
	\opn\revlex{revlex}
	\opn\div{div} \opn\Div{Div} \opn\cl{cl} \opn\Cl{Cl}
	\opn\Spec{Spec} \opn\Supp{Supp} \opn\supp{supp} \opn\Sing{Sing}
	\opn\Ass{Ass} \opn\Min{Min}\opn\Mon{Mon}
	\opn\Ann{Ann} \opn\Rad{Rad} \opn\Soc{Soc}
	\opn\Im{Im} \opn\Ker{Ker} \opn\Coker{Coker} \opn\Am{Am}
	\opn\Hom{Hom} \opn\Tor{Tor} \opn\Ext{Ext} \opn\End{End}
	\opn\Aut{Aut} \opn\id{id}
	
	\opn\nat{nat}
	\opn\pff{pf}%   \pf exists already
	\opn\Pf{Pf} \opn\GL{GL} \opn\SL{SL} \opn\mod{mod} \opn\ord{ord}
	\opn\Gin{gin} \opn\Hilb{Hilb}\opn\sort{sort}
	\opn\PF{PF}\opn\Ap{Ap}
	\opn\dist{dist}
	\opn\aff{aff}
	\opn\relint{relint} \opn\st{st}
	\opn\lk{lk} \opn\cn{cn} \opn\core{core} \opn\vol{vol}  \opn\inp{inp} \opn\nilpot{nilpot}
	\opn\link{link} \opn\star{star}\opn\lex{lex}\opn\set{set}
	\opn\width{wd}
	\opn\Fr{F}
	\opn\QF{QF}
	\opn\G{G}
	\opn\type{type}\opn\res{res}
	\opn\conv{conv}
	\opn\sr{sr}
	\opn\gr{gr}
	
	\def\pot#1#2{#1[\kern-0.28ex[#2]\kern-0.28ex]}

	\opn\dirlim{\underrightarrow{\lim}}
	\opn\inivlim{\underleftarrow{\lim}}
	\let\iso=\cong

	\let\to=\rightarrow
	
	\def\Implies{\ifmmode\Longrightarrow \else
		\unskip${}\Longrightarrow{}$\ignorespaces\fi}
	\def\implies{\ifmmode\Rightarrow \else
		\unskip${}\Rightarrow{}$\ignorespaces\fi}
	\def\iff{\ifmmode\Longleftrightarrow \else
		\unskip${}\Longleftrightarrow{}$\ignorespaces\fi}

	\let\:=\colon
	\newtheorem{Theorem}{Theorem}[section]

	\newtheorem{Example}[Theorem]{Example}

	\newtheorem{Conjecture}[Theorem]{Conjecture}

	\let\epsilon\varepsilon
	\let\kappa=\varkappa
	\def\qed{\ifhmode\textqed\fi
		\ifmmode\ifinner\hfill\quad\qedsymbol\else\dispqed\fi\fi}
	\def\textqed{\unskip\nobreak\penalty50
		\hskip2em\hbox{}\nobreak\hfill\qedsymbol
		\parfillskip=0pt \finalhyphendemerits=0}
	\def\dispqed{\rlap{\qquad\qedsymbol}}
	
	\opn\dis{dis}
	\def\pnt{{\raise0.5mm\hbox{\large\bf.}}}
	
	\opn\Lex{Lex}
	\opn\Max{Max}
	\opn\Shad{Shad}
	\opn\astab{astab}

	\def\m{\mathfrak{m}}
	
	\opn\v{v}
	
\begin{document}

	\title{The eventual shape of the Betti table of $\m^kM$}
	\author{Antonino Ficarra, J\"urgen Herzog, Somayeh Moradi}
	%\date{\today}
	
	\address{Antonino Ficarra, Department of mathematics and computer sciences, physics and earth sciences, University of Messina, Viale Ferdinando Stagno d'Alcontres 31, 98166 Messina, Italy}
	\email{antficarra@unime.it}
	
	\address{J\"urgen Herzog, Fakult\"at f\"ur Mathematik, Universit\"at Duisburg-Essen, 45117 Essen, Germany} \email{juergen.herzog@uni-essen.de}
	
	\address{Somayeh Moradi, Department of Mathematics, Faculty of Science, Ilam University, P.O.Box 69315-516, Ilam, Iran}
	\email{so.moradi@ilam.ac.ir}
	
	\thanks{The third author is supported by the Alexander von Humboldt Foundation.
	}
	
	
	
	\maketitle
	
	\begin{abstract}
	Let $S$ be the polynomial ring  over a field  $K$  in a finite set of variables, and let $\mm$ be the graded maximal ideal of  $S$.  It is known that for  a finitely generated graded $S$-module $M$ and all integers  $k\gg 0$, the module $\m^kM$ is componentwise linear. For large $k$ we describe the pattern of the Betti table of $\mm^kM$ when $\chara(K)=0$ and $M$ is a submodule of a finitely generated graded free $S$-module. Moreover, we show that for any $k\gg 0$, $\m^kI$ has linear quotients if $I$ is a monomial ideal.
	\end{abstract}

		\section*{Introduction}
		
Let $K$ be a field and $S=K[x_1,\dots,x_n]$ be the polynomial ring. We denote by $\m$ the unique graded maximal ideal of $S$. 
%This note is inspired by a result of Eisenbud and Goto \cite{EG},   who gave a  formula for  the regularity of a finitely generated graded $S$-module $M$, namely 
%\[
%\reg M=\min\{i\: M_{\geq i} \text{ has linear resolution}\}.
%\]
%Here $M_{\geq i}=\Dirsum_{j\geq i}M_j$ is the $i$th truncation  module  of $M$. 
For a finitely generated graded $S$-module $M$, it follows from \cite{EG} and \cite{R}   and independently from \cite[Theorem 3.2]{Se}  that for large enough $k$, $\mm^k M$ is componentwise linear. We denote the smallest such number by $c_M$. 
%Theorem 1.3 of the paper quickly follows from the fact that a finitely generated graded module N is componentwise linear if and only if gr_m(N) has a linear resolution. This is contained already in the work of Tim Roemer, and reproduced in the paper of Herzog and Iyengar on Koszul modules. Indeed, consider the module N=gr_m(M). By Eisenbud-Goto, (gr_m(M)_<d> = gr_m(m^dM) has a linear resolution for all d >= reg gr_m(M). Hence m^dM is componentwise linear for all n>=reg gr_m(M).
% then $M_{\geq i}=\mm^{i-d}M$ for all $i\geq d$. Therefore it is natural to ask what can be said about the Betti-diagram of $\mm^k M$,  when $M$ is not necessarily generated in a single degree. Most closest to modules with linear resolution are those which are componentwise linear, and indeed it turns out  that the modules $\mm^kM$ are componentwise linear for $k\gg 0$. This is shown in Theorem~\ref{asymptotic}, where also  the  lowest power $c_M$ is determined from which  on $\mm^kM$ is componentwise linear. 
In this paper we describe the pattern of the Betti table of $\mm^kM$ when $\chara(K)=0$ and $M$ is a submodule of a finitely generated graded free $S$-module. We show that if $k\geq c_M+1$, then all non-zero  strands of the Betti table are full in the sense that  if $\beta_{0,j}(\mm^kM)\neq 0$, then $\beta_{i,i+j}(\mm^kM)\neq 0$ for $i=0,\ldots,n-1$. For the proof we use the result of Nguyen \cite[Theorem 3.5 and its proof]{Nguyen}, which implies that the  maps $\Tor_i^S(K,\m^{k}M)\to \Tor_i^S(K,\m^{k-1}M)$  are zero for all $i$ when $k\geq c_M+1$. Furthermore we use that for a componentwise linear module $M$ which is a submodule of a finitely generated graded free $S$-module,  $\beta_{i, i+j}(M)\neq 0$ implies that $\beta_{i', i'+j}(M)\neq 0$ for all $0\le i'\le i$. For this fact, which follows from   %\cite[Theorem 8.2.22, Theorem 8.2.23]{JT},
\cite[Lemma 3.8, Theorem 3.10]{CR}, it is required that $\chara(K)=0$. We also show that for all  $k\geq c_M$, the Betti tables of $\m^kM$ have the same pattern. In other words, the $\ell$th strand of $\mm^kM$ is non-zero, if and only if the $(\ell+1)$th strand of $\mm^{k+1}M$ is non-zero. When  $\depth \gr_{\mm}(M)>0$, we determine precisely which strands are non-zero.

Finally,  for the more special case that $I$ is a monomial ideal, we show  in Theorem~\ref{linquot} that $\mm ^kI$ has linear quotients for all large $k$. For monomial ideals this is an improvement of Theorem~\ref{asymptotic}, since any monomial ideal with linear quotients is componentwise linear.

		\section{The eventual shape of the Betti table of $\m^kM$}
		
		Let $M=\bigoplus_{i\in \ZZ}M_i$ be a finitely generated graded $S$-module.  For any integer $d\ge0$, we denote by $M_{\langle d\rangle}$ the graded $S$-module generated by the $K$-vector space $M_d$, and by $M_{\ge d}$ we denote the truncated module $\bigoplus_{i\ge d}M_i$. The \textit{initial degree} of $M$, denoted by $\alpha(M)$, is the smallest integer $j$ such that $M_{j}\ne0$.

The next result is well-known. For the convenience of the reader a short proof is provided.

	\begin{Theorem}\label{asymptotic}
	Let $M$ be a finitely generated graded $S$-module. Then $\m^kM$ is componentwise linear for all $k\ge  \reg(\gr_{\mm}(M))$.
\end{Theorem}

\begin{proof}
By \cite[Theorem 3.2.8]{R} a finitely generated graded module $N$ is componentwise linear if and only if $\gr_{\mm}(N)$ has a linear resolution.  Therefore by Eisenbud-Goto, $(\gr_{\mm}(M))_{\langle d\rangle} = \gr_{\mm}(\mm^dM)$ has a linear resolution for all $d \geq \reg(\gr_{\mm}(M))$, see \cite[Theorem 4.3.1]{BH}. Hence $\mm^dM$ is componentwise linear for all $n\geq \reg( \gr_{\mm}(M))$.	
\end{proof}

In the next result we describe the eventual shape of the Betti table of $\mm^k M$. 

\begin{Theorem}\label{pattern}
Let the base field $K$ be of characteristic zero, and let $M$ be a graded $S$-module which is a submodules of a finitely generated graded free $S$-module. Then
\begin{enumerate}
\item[\textup{(a)}] For any $k\ge c_M+1$, all non-zero strands of $\m^kM$ are full, which means that if the $\ell$th strand of $\m^kM$ is non-zero, then $\beta_{i,i+\ell}(\m^kM)\ne0$ for all $i=0,\dots,n-1$.
\item[\textup{(b)}] For $k\ge c_M$, the $\ell$th strand of $\m^{k} M$ is non-zero if and only if the $(\ell+1)$th strand of $\m^{k+1} M$ is non-zero.
\item[\textup{(c)}] For all $k\ge c_M+1$, the Betti diagrams of the modules $\m^k M$ have the same shape.

\item[\textup{(d)}]  Suppose that $M$ is minimally generated in degrees $d_1,\ldots,d_m$. For any $k\ge c_M$, if the $\ell$th strand of $\m^{k} M$ is non-zero, then $\ell-k=d_i$ for some $1\leq i\leq m$.
 If in addition $\depth \gr_{\mm}(M)>0$, then  the $\ell$th strand of $\m^{k} M$ is non-zero if and only if $\ell-k=d_i$ for some $1\leq i\leq m$.
%\item[\textup{(d)}] If $\depth(M)>0$, then the statements \textup{(a)} and \textup{(c)} hold for $k\ge c_M+1$ and the statement \textup{(b)} holds for $k\ge c_M$.    
\end{enumerate}	
\end{Theorem}		

\begin{proof}
(a) For any positive integer $k$, we set $W_k=\m^kM/\m^{k+1}M$. Then the short exact sequence $0\to \m^{k}M\to \m^{k-1}M\to W_{k-1}\to 0$ gives rise to the long exact sequence 
\begin{eqnarray*}
\cdots&\to& \Tor_i^S(K,\m^{k}M)\to \Tor_i^S(K,\m^{k-1}M) \to \Tor_i^S(K,W_{k-1})\\ &\to&
 \Tor_{i-1}^S(K,\m^{k}M)\to\cdots.
\end{eqnarray*}
Let $k\geq c_M+1$. Then by Theorem~\ref{asymptotic} the module
$\m^{k-1}M$ is componentwise linear. Therefore \cite[Theorem 3.5 and its proof]{Nguyen} (see also \cite[Theorem 3.10]{NV}) imply that the maps $\Tor_i^S(K,\m^{k}M)\to \Tor_i^S(K,\m^{k-1}M)$ in the above long exact Tor-sequence are zero for all $i$. Thus we obtain the short exact sequences
\begin{equation}\label{vanishing}
0\to \Tor_i^S(K,\m^{k-1}M) \to \Tor_i^S(K,W_{k-1}) \to \Tor_{i-1}^S(K,\m^{k}M)\to 0
\end{equation}
for all $i\geq 0$.  %$\Tor_{n+1}^S(K,\m^{k-1}M)=\Tor_{n+1}^S(K,W_{k-1})=0$, the  exactness of (\ref{vanishing})  implies that 
We claim that the $\ell$th strand of $\m^k M$ is non-zero if and only if $(W_{k-1})_{\ell-1}\neq 0$, and when this is the case, the $\ell$th strand of $\m^k M$ is full. 
Indeed, since $\depth \m^{k-1}M>0$, we have $\Tor_{n}^S(K,\m^{k-1}M)=0$. 
Thus from the exact sequence (\ref{vanishing}) we obtain
$\Tor_{n-1}^S(K,\m^{k}M)\iso \Tor_n^S(K,W_{k-1})$.  
Since $W_{k-1}$ is a $K$-vector space, we conclude that $(W_{k-1})_{\ell-1}\neq 0$ if and only if $\Tor_i^S(K,W_{k-1})_{i+\ell-1}\neq 0$ for all $1\leq i\leq n$. 
Therefore, whenever $(W_{k-1})_{\ell-1}\neq 0$, we get $\beta_{n-1,n-1+\ell}(\m^k M)=\dim_K \Tor_{n-1}^S(K,\m^{k}M)_{n+\ell-1}=\dim_K \Tor_n^S(K,W_{k-1})_{n+\ell-1}\neq 0$.  
Since $\m^{k}M$ is componentwise linear, this together with \cite[Lemma 3.8, Theorem 3.10]{CR} implies that the $\ell$th strand of $\mm^{k}M$ is full.    
 
Conversely, suppose that the $\ell$th strand of $\m^{k}M$ is non-zero. This means that $\Tor_s^S(K,\m^{k}M)_{s+\ell}\neq 0$ for some $1\leq s\leq n-1$. Hence the exactness of the sequence (\ref{vanishing}) implies that $ \Tor_{s+1}^S(K,W_{k-1})_{s+\ell}\neq 0$. As it was mentioned above, this is equivalent to $(W_{k-1})_{\ell-1}\neq 0$ and when this happens the $\ell$th strand of $\m^{k}M$ is full.   

\medskip

(b) Let $k\ge c_M$. Since $\m^{k}M$ is componentwise linear, applying once again \cite[Lemma 3.8, Theorem 3.10]{CR} we have the $\ell$th strand of $\m^{k}M$ is non-zero if and only if $\Tor_0^S(K,\m^{k}M)_{\ell}\neq 0$. 
Note that $\Tor_0^S(K,\m^{k}M)\iso W_{k}$. So
$\Tor_0^S(K,\m^{k}M)_{\ell}\neq 0$ if and only if $(W_{k})_{\ell}\neq 0$.  By the proof of (a) this is the case if and only if the $(\ell+1)$th strand of $\m^{k+1} M$ is non-zero. 

(c) follows from (a) and (b). 

(d) By \cite[Proposition 3.2]{HR}, if $N$ is a componentwise linear $S$-module, then $\beta_{i,i+j}(N)=\beta_i(N_{\langle j\rangle})-\beta_i(\mm N_{\langle j-1\rangle})$ for all $i$ and $j$. If $N$ has no minimal homogeneous generator of degree $j$, then $N_{\langle j\rangle}=\mm N_{\langle j-1\rangle}$, and hence $\beta_{i,i+j}(N)=0$ for all $i$. This shows that the $\ell$th strand of $N$ is non-zero if and only if $\ell$ is a degree of a minimal homogeneous generator of $N$. Now, let  $k\ge c_M$. Since $\mm^k M$  has a minimal set of homogeneous generators, such that the degrees of its elements belong to the set $\{k+d_1,\ldots,k+d_m\}$, the first statement follows from the above fact.

Now, suppose that  $\depth \gr_{\mm}(M)>0$. Since $K$ is infinite, we may choose a non-zerodivisor element say $x$ of $\gr_{\mm}(M)$ of degree one. Thus the multiplication maps $\mm^kM/\mm^{k+1}M\to \mm^{k+1}M/\mm^{k+2}M$ induced by $x$ are injective for all $k\ge0$. Hence, if $f$ is a minimal generator of $M$ having degree $d_i$ for some $i$, then $x^k f$ is a minimal generator of $\mm^k M$ of degree $k+d_i$. If $k\ge c_M$, then for all $i$ the $(k+d_i)$th strand of $\mm^k M$ is non-zero. The first statement shows that these are indeed the only non-zero strands of $\mm^k M$.
\end{proof}

The assumption $\depth \gr_{\mm}(M)>0$ given in part (d) can not be dropped. Indeed, the componentwise linear ideal $I=(x^3,x^2y^2,y^3)$ of $S=K[x,y]$ is generated in two degrees, but $\mm I=(x^4,x^3y,xy^3,y^4)$ is generated in a single degree. Note that $\depth \gr_{\mm}(I)=0$.

We demonstrate Theorem \ref{pattern} with an example.
\begin{Example}
Consider the graded ideal
$I=(x_1x_2^3+x_3^4,x_1+x_2+x_4,x_2^3)$ in the polynomial ring $S=K[x_1,x_2,x_3,x_4]$. The Betti diagrams of the ideals $I$, $\mm I$, $\mm^2 I$ and $\mm^3 I$ are demonstrated below. It can be seen that after the second power of $\mm$, the non-zero strands of $\mm^k I$ are full and have the same shape. Moreover, the non-zero strands of $\mm^k I$ appear in the degrees $k+1,k+3$ and $k+4$ for $k\geq 2$. 

\begin{minipage}{0.4\textwidth} 
	%\fontsize{12pt}{14pt}\selectfont 
	\begin{verbatim} 
	       0    1    2
	--------------------
	1:     1    -    -
	2:     -    -    -
	3:     1    1    -
	4:     1    1    -
	5:     -    -    -
	6:     -    1    1
	--------------------
	Tot:   3    3    1
	\end{verbatim}
\end{minipage}
\hfill
\begin{minipage}{0.7\textwidth}   
	\begin{verbatim}
       0    1    2    3
-------------------------
2:     4    6    4    1
3:     -    -    -    -
4:     3    6    4    1
5:     3    6    4    1
6:     -    1    1    -
-------------------------
Tot:   10   19   13    3
		
	\end{verbatim}
\end{minipage}
\newline

\begin{minipage}{0.4\textwidth}  
	\vskip 5mm
	\hskip 22mm	$I$ 
	\begin{verbatim}
		
       0    1    2    3
-------------------------
3:    10   20   15    4
4:     -    -    -    -
5:     6   14   11    3
6:     6   15   12    3
-------------------------
Tot:   22   49   38   10
		
	\end{verbatim}
\end{minipage}
\hfill
\begin{minipage}{0.6\textwidth} 
	\hskip 24mm	$\mm I$   
	\begin{verbatim}
				
       0    1    2    3
-------------------------
4:    20   45   36   10
5:     -    -    -    -
6:    10   25   21    6
7:     9   24   21    6
-------------------------
Tot:   39   94   78   22
	\end{verbatim}
\end{minipage}

\vskip 1mm \hskip 22mm $\mm^2 I$ \hskip 56mm  $\mm^3 I$
\end{Example}

\medskip
\medskip
Let $I$ be a graded ideal. Note that by \cite[Corollary 8.2.14]{JT}, $\reg(\mm^{c_I} I)$ is equal to the highest degree of a generator in a minimal set of
generators of $\mm^{c_I} I$, which is equal to $\max\{j:\  (W_{c_I})_j\neq 0\}$. 
So Theorem~\ref{pattern} implies that for any $k\geq c_I$, $\reg(\mm^k I)=(k-c_I)+\max\{j:\  (W_{c_I})_j\neq 0\}$.  
For $1\leq k\leq c_I$,
the examples show an interesting result, that the regularity is the same as that of $\mm^{c_I} I$.  In other words we expect

\begin{Conjecture}
	Let $I$ be a  graded ideal. Then for any $1\leq k\leq c_I$,  $\reg(\mm^k I)=\max\{j:\  (W_{c_I})_j\neq 0\}$. 
\end{Conjecture}

Let $I$ be a monomial ideal with the minimal set of monomial generators $G(I)$. The ideal $I$ is said to have {\em linear quotients}, if there exists an order 
$u_1<\cdots<u_r$ on $G(I)$ such that $(u_1,\ldots,u_{i-1}):u_i$ is generated by some variables. Such an order is called an {\em admissible order} for $I$.

To a given a monomial ideal $I$ and an order $\mathcal{O}: u_1<\cdots<u_r$ on $G(I)$ we attach numerical invariants as follows. For each $2\leq i\leq r$, let $(u_1,\ldots,u_{i-1}):u_i=(w_{i,1},\ldots,w_{i,\ell_i})$. We set $\lambda_{I,\mathcal{O},u_i}=\sum_{j=1}^{\ell_i}(\deg(w_{i,j})-1)$. One can easily see that $I$ has linear quotients with respect to the order $\mathcal{O}$  if and only if $\lambda_{I,\mathcal{O},u_i}=0$ for all $2\leq i\leq r$. For two monomials $u$ and $v$, we set $u:v=u/\gcd(u,v)$. Furthermore, the \textit{support} of $u$ is the set $\text{supp}(u)=\{i:x_i\ \text{divides}\ u\}$.

\begin{Theorem}\label{linquot}
	Let $I\subset S$ be a monomial ideal. Then  $\mm^k I$ has linear quotients for $k\gg 0$.  	
\end{Theorem}

\begin{proof}
	Let $G(I)=\{u_1,\ldots,u_r\}$ with $\deg(u_1)\leq \deg(u_2)\leq\cdots\leq\deg(u_r)$. We consider the order $\mathcal{O}:u_1<\cdots<u_r$. For each $2\leq i\leq r$, let $(u_1,\ldots,u_{i-1}):u_i=(w_{i,1},\ldots,w_{i,\ell_i})$, and let  $A_i=\bigcup_{j=1}^{\ell_i}\supp(w_{i,j})$. Then we may write $[n]=\{s_{i,1},\ldots,s_{i,n}\}$, where $A_i=\{s_{i,1},s_{i,2},\ldots,s_{i,n_i}\}$ for some $n_i\leq n$. Moreover, we set $s_{1,j}=j$ for $1\leq j\leq n$ and $n_1=n$. 
	%Consider any order on the generators $x_iu_j$, $1\leq i\leq n$ and $1\leq j\leq r$ of $\mm I$ which satisfies the following conditions.  
	%\begin{itemize}
	%	\item If $i<j$, then $x_ru_i<x_su_j$.
	%	\item If $i=j$, $x_r\in A_i$ and $x_s\notin A_i$, then $x_ru_i<x_su_i$.
	%\end{itemize}
	
	Consider the following order $\mathcal{O}'$ on the monomials $x_iu_j$, $1\leq i\leq n$ and $1\leq j\leq r$:
	\begin{eqnarray*}\label{linear q}
		&x_{s_{1,1}}u_1<x_{s_{1,2}}u_1<\cdots<x_{s_{1,n}}u_1<\\
		&x_{s_{2,1}}u_2<x_{s_{2,2}}u_2<\cdots<x_{s_{2,n}}u_2<\\
		&\vdots\\
		&<x_{s_{r,1}}u_r<x_{s_{r,2}}u_r<\cdots<x_{s_{r,n}}u_r.
	\end{eqnarray*}
	
	For any $i$ and $j$, let $f_{i,j}=x_{s_{i,j}}u_i$ and $$L_{i,j}=(f_{i',j'}:\ \textrm{ $i'<i$, or $i'=i$ and $j'<j$}):f_{i,j}.$$
	If $j>n_i$, then we have $L_{i,j}=(x_{s_{i,\ell}}:\  1\leq \ell\leq j-1)$.
	Indeed, $f_{i,\ell}:f_{i,j}=x_{s_{i,\ell}}$ for all $\ell<j$. Moreover, by our choice of $A_i$, for any $i'<i$, the monomial $u_{i'}:u_i$ is divided by $x_{s_{i,t}}$ for some $1\leq t\leq n_i$. Hence so does $f_{i',j'}:f_{i,j}$, noting that $j>n_i$. 
	
	Now, suppose that $j\leq n_i$.  Then $x_{s_{i,j}}$ divides $w_{i,k}$ for some $k$. Without loss of generality we may assume that $x_{s_{i,j}}$ divides $w_{i,k}$ if and only if $1\leq k\leq q_i$  for some $q_i\leq \ell_i$. Fix an integer $1\leq k\leq q_i$. We have $w_{i,k}=u_{i'}:u_i$ for some $i'<i$. Since $u_i$ and $u_{i'}$ are minimal monomial generators of $I$, there exists a variable $x_p$ which divides $u_i:u_{i'}$. It follows that $x_pu_{i'}:x_{s_{i,j}}u_i=u_{i'}:x_{s_{i,j}}u_i=w_{i,k}/x_{s_{i,j}}$. Hence   $w_{i,k}/x_{s_{i,j}}\in L_{i,j}$ for $1\leq k\leq q_i$. Also for $k>q_i$, with $w_{i,k}=u_{i'}:u_i$, we have  $w_{i,k}=x_{s_{i,j}}u_{i'}:f_{i,j}\in L_{i,j}$. Hence 
	\begin{eqnarray}\label{order}
		\ \ \ \ \ \ \ \ (w_{i,k}/x_{s_{i,j}}: 
		1\leq k \leq q_i)+(w_{i,k}: q_i<k<\ell_i)+(x_{s_{i,\ell}}:\  1\leq \ell\leq j-1)\subseteq L_{i,j}.
	\end{eqnarray}
	To see that $(\ref{order})$ is indeed an equality, it is enough to show that each monomial $f_{i',j'}:f_{i,j}\in L_{i,j}$  is divided by one of the monomials in the left-hand side ideal in (\ref{order}). If $i'=i$, then obviously 
	$f_{i',j'}:f_{i,j}=x_{s_{i,j'}}$.  
	Let $i'<i$. Then $u_{i'}:u_i=w_{i,k} v$ for some integer $k$ and a monomial $v$. %If $x_{s_{i,j}}$ divides $u_{i'}:u_i$, then $(u_{i'}:u_i)/x_{s_{i,j}}$ divides $f_{i',j'}:f_{i,j}$. 
	If $1\leq k \leq q_i$, then $w_{i,k}/x_{s_{i,j}}$ divides $f_{i',j'}:f_{i,j}$. Otherwise $w_{i,k}$ divides $f_{i',j'}:f_{i,j}$, as desired.
	%If $x_{s_{i,j}}$ does not divide $u_{i'}:u_i$, then $f_{i',j'}:f_{i,j}$ is divided by $u_{i'}:u_i$, which implies that is divided by $w_{i,k}$ and $q_i<k<\ell_i$. 
	
	In the given order, some of the generators $f_{i,j}$ may not be minimal or may be repeated in the order. Note that if it happens that $f_{i,j}$ is not a minimal generator, it is divided by some $f_{i',j'}$ with $i'<i$, since  $\deg(u_1)\leq \deg(u_2)\leq\cdots\leq\deg(u_r)$.  We remove all the non-minimal monomial generators $f_{i,j}$ and also the repeated ones wherever they appear again. Then considering the order induced by $\mathcal{O}'$  on the minimal generators $f_{i,j}$, we get an order, say $\mathcal{O}_1$ on the minimal monomial generators of $\mm I$. It can be easily seen that for any minimal generator  $f_{i,j}\in \mm I$, $$(f_{i',j'}: f_{i',j'}<f_{i,j}  \textrm{ in the order } \mathcal{O}_1)=(f_{i',j'}: f_{i',j'}<f_{i,j}  \textrm{ in the order } \mathcal{O}').$$ Hence   $(f_{i',j'}: f_{i',j'}<f_{i,j}  \textrm{ in the order } \mathcal{O}_1):f_{i,j}=L_{i,j}$.
Comparing the generators of $(u_1,\ldots,u_{i-1}):u_i$ and $L_{i,j}$, one can see that $\lambda_{\mm I,\mathcal{O}_1,f_{i,j}}\leq \lambda_{I,\mathcal{O},u_i}$ and  if $\lambda_{I,\mathcal{O},u_i}>0$, then $\lambda_{\mm I,\mathcal{O}_1,f_{i,j}}< \lambda_{I,\mathcal{O},u_i}$. 
	
We rename the set $G(\mm I)$ as $\{g_1,\ldots,g_m\}$. Notice that the order $\mathcal{O}_1$ is a degree increasing order on $G(\mm I)$. Then one may apply the same approach as above on this order to get an order $\mathcal{O}_2$ on $G(\mm^2 I)$, so that  
 $\lambda_{\mm^2 I,\mathcal{O}_2,x_jg_i}\leq \lambda_{\mm I,\mathcal{O}_1,g_i}$ and if $\lambda_{\mm I,\mathcal{O}_1,g_i}>0$, then $\lambda_{\mm^2 I,\mathcal{O}_2,x_jg_i}< \lambda_{\mm I,\mathcal{O}_1,g_i}$. 
 Proceeding this way, for each $k$	we obtain an order $\mathcal{O}_k$ on $G(\mm^k I)$, and it follows that there exists a positive integer $t$ such that $\lambda_{\mm^{t} I,\mathcal{O}_t,h_i}=0$ for any $h_i\in G(\mm^t I)$. Hence $\mathcal{O}_t$ is an admissible order for $\mm^t I$. Obviously if $\mathcal{O}_t$ is an admissible order for $\mm^t I$,  $\mathcal{O}_k$ is an admissible order for $\mm^k I$ for all $k\geq t$.      
\end{proof}

	\end{document}